\documentclass[a4paper,12pt]{article}

\usepackage{color,makeidx,bm,latexsym,amscd,amsmath,amssymb,enumerate,stmaryrd,amsthm,float,graphicx,psfrag,geometry}

\usepackage[all]{xy}

\usepackage{tikz,epic,eso-pic}

\theoremstyle{plain}
\newtheorem{theorem}{Theorem}[section]
\newtheorem{corollary}[theorem]{Corollary}

\theoremstyle{definition}
\newtheorem{definition}[theorem]{Definition}
\newtheorem{remark}[theorem]{Remark}
\newtheorem{example}[theorem]{Example}
\newtheorem{proposition}[theorem]{Proposition}

\newcommand{\Q}{\mathbb{Q}}

\newcommand{\scS}{\mathcal{S}}
\newcommand{\Z}{\mathbb{Z}}

\newcommand{\MH}{\operatorname{MH}} % magnitude homology
\newcommand{\MC}{\operatorname{MC}} % magnitude chain complex
\newcommand{\rank}{\operatorname{rank}}
\newcommand{\Ker}{\operatorname{Ker}}
\renewcommand{\Im}{\operatorname{Im}}
\newcommand{\id}{\operatorname{id}}

\title{Magnitude homology of graphs and discrete Morse theory on 
Asao-Izumihara complexes} 
\author{Yu Tajima\thanks{Department of Mathematics, Graduate School of Science, 
Hokkaido University, North 10, West 8, Kita-ku, 
Sapporo 060-0810, JAPAN 
E-mail: tajima@math.sci.hokudai.ac.jp}, 
Masahiko Yoshinaga\thanks{Department of Mathematics, Graduate School of Science, Osaka University, Toyonaka 560-0043, JAPAN 
E-mail: yoshinaga@math.sci.osaka-u.ac.jp}}
\date{\today}
\begin{document}
\maketitle

\begin{abstract} 
Recently, Asao and Izumihara introduced CW-complexes whose 
homology groups are isomorphic to direct summands of the graph magnitude 
homology group. 
In this paper, we study the homotopy type of the CW-complexes 
in connection with the diagonality of magnitude homology groups. 
We prove that the Asao-Izumihara complex is homotopy equivalent to 
a wedge of spheres for pawful graphs introduced by Y. Gu. 
The result can be considered as a homotopy type version of Gu's result. 
We also formulate 
a slight generalization of the notion of pawful graphs and find new non-pawful 
diagonal graphs of diameter $2$. 
\end{abstract}

%\tableofcontents

\section{Introduction}%%%%%%%%%%%%%%%%%%%%%%%%%%%%%%%%%%%%
\label{sec:intro}

Magnitude is an invariant for metric spaces introduced by Leinster \cite{Lei-met}, 
which measures the number of efficient points. 
As a categorification of the magnitude, Hepworth and Willerton 
defined \emph{magnitude homology} for graphs \cite{HW}, and later 
Leinster and Shulman generalized to metric spaces (furthermore, for enriched 
categories) \cite{lei-shu}. 
Several techniques 
to compute the magnitude homology groups have been developed 
\cite{HW, bot-kai, KY, gomi, G, saz-sum}. 

Let $G$ be a graph. Let $k, \ell\geq 0$. 
Denote the magnitude homology of degree $k$ with length $\ell$ by 
$\MH_k^\ell(G)$ (see \S \ref{sec:pre} for details). 
The graph $G$ is called diagonal if the magnitude homology vanishes when 
$k\neq \ell$ (Definition \ref{def:diagonal}). 
The diagonality of a graph is an important feature, because 
the rank of magnitude homology is determined solely by the magnitude of $G$ if 
$G$ is diagonal. 
There are several recent studies on the diagonality \cite{ahk, bot-kai, G, HW, saz-sum}. 
Among others, let us recall two results concerning diagonality. 
The first one is a result by Asao, Hiraoka and Kanazawa \cite[Theorem 1.5]{ahk}, 
which asserts that if $G$ is diagonal and is not a tree, 
then every edge is contained in a cycle of length at most $4$. 
This gives a necessary condition for a graph to be diagonal. 
The second one is concerning a class of diagonal graphs with diameter $2$. 
Gu \cite{G} introduced the notion of \emph{pawful graphs} 
(see Definition \ref{def:pawful}) and proved that pawful graphs are diagonal. 
Thus we have the following inclusions. 
\[
\{\mbox{Pawful graphs}\}\subseteq\{\mbox{Diagonal graphs}\}
\subseteq
\{\mbox{Graphs of Asao-Hiraoka-Kanazawa}\}. 
\]
Pawful graphs are graphs with diameter at most $2$ satisfying some conditions. 
Many previously known diagonal graphs of diameter $2$ 
are pawful. 
However, even for graphs of diameter $2$, both of inclusions above are proper. 
We will exhibit non-pawful diagonal graphs (Example \ref{ex:01} and Example \ref{ex:02}) 
and a 
non-diagonal graph with diameter $2$ such that each edge is 
contained in a cycle of length at most $4$ (Example \ref{ex:03}). 

On the other hand, recently, Asao and Izumihara 
\cite{AI} constructed a pair of simplicial complexes 
$K_\ell(a, b)\supset K'_\ell(a, b)$ for two vertices $a, b$ of a graph $G$ 
and $\ell\geq 3$ such that 
the homology of the quotient $K_\ell(a, b)/K'_\ell(a, b)$ is isomorphic to a 
direct summand of $\MH_*^\ell(G)$ up to degree shift. 
It is straightforward that a graph $G$ is diagonal if and only if 
the homology of the Asao-Izumihara complex $K_\ell(a, b)/K'_\ell(a, b)$ 
concentrates at the top degree (Proposition \ref{prop:diagonal}). 
The purpose of this paper is to study the topology of 
Asao-Izumihara complexes in connection with the diagonality of $G$. 
In particular, we prove that the Asao-Izumihara complex of 
a pawful graph is homotopy equivalent to the wedge of spheres. 
Our result may be considered 
as a homotopy-type version of Gu's result. 
The proof uses discrete Morse theory. 
%First of all, 
%Such a topological behavior suggests that 
%for a diagonal graphs, Asao-Izumihara complexes are homotopy 
%equivalent a wedge of spheres. 

%The main purpose of the present paper is to refine Gu's result by establishing 
%the results at the level of homotopy types of certain spaces. 

%A graph is defined by a set of vertices and a set of edges, and it can be considered as a metric space. Magnitude homology of graphs is defined by Hepworth and Willerton, as categorification of magnitude. In general, it is difficult to compute magnitude homology of graphs. Several techniques computing magnitude homology has been developed. 

%Hepworth and Willerton proved that the join of non-empty graphs is diagonal in \cite{HW}.
%Pawful graphs are defined in \cite{G} as a new class of diagonal graphs. Using discrete Morse theory, we prove that the Asao-Izumihara complex corresponding to pawful graphs are homotopy equivalent to wedge of spheres .

This paper is organized as follows. In \S\ref{sec:pre}, first we recall the definition of magnitude and magnitude homology of graphs, and introduce Asao-Izumihara complex following \cite{AI}. Next, we recall the definition of an acyclic matching and several results in discrete Morse theory following \cite{K}. In \S\ref{sec:main}, We describe the homotopy type of the Asao-Izumihara complex corresponding to pawful graphs. In \S\ref{sec:proofs}, we give the proof of the main result using discrete Morse theory. In \S\ref{sec:general}, we introduce a class of diagonal graphs which properly contains pawful graphs.

\section{Background}
\label{sec:pre}
\subsection{Magnitude and magnitude homology of graphs}%%%%%%%%%%%%%%%%%%%

Let $G=(V(G), E(G))$ be a graph, 
where $V(G)$ is the set of vertices and $E(G)$ is the set of edges. 
In this paper, we assume that a graph is finite, simple, and connected. 
We call a sequence of consecutive edges 
$\{v_0, v_1\}, \{v_1, v_2\}, \cdots, \{v_{k-1}, v_k\}\in E(G)$ $(v_0, v_1, \cdots, v_k\in V(G))$ a path of length $k$. 
The distance function 
$d_G\colon V(G)\times V(G)\rightarrow \mathbb{Z}_{\geq 0}$ 
is defined as the minimum length of paths connecting two vertices. 
%on graph $G$ is defined by
%\[
%(a, b)\mapsto\begin{cases}\min\left\{k\in\mathbb{Z}_{\geq  0}\left|
%\begin{array}{l}
%\exists e_1, e_2, \cdots, e_k\text{(sequence of edges)}\in E(G)\\ \text{ such that } a\in e_1, b\in e_k
%\end{array}
%\right.\right\}&(a\neq b), \\0&(a=b).\end{cases}
%\]

\begin{definition}[\cite{L}, Definition 2.1]
Let $G$ be a graph. Let $q$ be a variable. 
Define the $|V(G)|\times |V(G)|$ matrix $Z_G$ with entries in 
$\Q(q)$ by  $Z_G:=\Bigl(q^{d(x, y)}\Bigr)_{x, y\in V(G)}.$ 
Note that $Z_G$ is invertible, and let us denote the 
entry of the inverse matrix $Z_G^{-1}$ by $Z_G^{-1}(x, y)$. 
The magnitude of $G$, denoted by $\#G$, is defined as follows:
\[
\#G:=\sum_{x, y\in V(G)} Z_G^{-1}(x, y)\in \Q(q).
\]
\end{definition}

Let $G$ be a graph. We say that $(x_0, x_1, \cdots, x_k)\in V(G)^{k+1}$ 
is a sequence if $x_i\neq x_{i+1}$ for any $i\in\{0, 1, \cdots, k-1\}$. 
Let $x=(x_0, x_1, \cdots, x_k)\in V(G)^{k+1}$ be a sequence, we say that 
$y=(y_0, y_1, \cdots, y_{k'})\in V(G)^{k'+1}$ is a subsequence of $x$ 
(denote $y\prec x$) if there exist indices $0=i_0<i_1<\cdots <i_{k'}=k$ 
such that $y_j=x_{i_j}$ for any $j\in\{0, 1, \cdots, k'\}$.
Length of $x$ is defined as $L(x):=d(x_0, x_1)+d(x_1, x_2)+\cdots+d(x_{k-1}, x_k).$ 
We call $x_i$ a smooth point if $L(x)=L((x_0, \cdots, \widehat{x}_i, \cdots, x_k))$.

\begin{definition}[Magnitude homology of graphs]
\label{def:MC}
Fix $\ell\geq 0$.
Define the abelian group $\MC_k^\ell(G)$ and the map $\partial$ as follows.
\[
\begin{split}
&\MC_k^\ell(G):=\bigoplus_{\,x\in\{x=(x_0, \cdots, x_k)\in V(G)^{k+1}\colon\text{sequence }\mid L(x)=\ell\,\}} \mathbb{Z}\langle x\rangle, \\
&\partial\colon \MC_k^\ell(G)\rightarrow \MC_{k-1}^\ell(G),\ \partial:=\sum_{i=1}^{k-1}(-1)^i\partial_i, \\
&\partial_i(x_0, \cdots, x_k):=\begin{cases}(x_0, \cdots, \widehat{x_i}, \cdots, x_k)&
(\text{if } x_i\text{ is a smooth point of }x), \\0&(\text{otherwise}).\end{cases}\\
\end{split}
\]
Then $(\MC_*^\ell(G), \partial)$ is a chain complex and it is called the magnitude chain complex. The magnitude homology of $G$ is defined as the homology of the chain 
complex: $\MH_k^\ell(G):=H_k(\MC_*^\ell(G))$.
\end{definition}

\begin{definition}
\label{def:diagonal}
A graph $G$ is called \emph{diagonal} if $\MH_k^\ell(G)=0$ for $k\neq \ell$. 
\end{definition}

For example, trees, complete graphs and  the join of non-empty graphs are known as diagonal graphs (\cite{HW}).

\begin{theorem}[\cite{HW}, Theorem 8.]
Let $G$ be a graph. Then $\#G$ has the following $($Taylor$)$ expansion.
\[
\#G=\sum_{\ell\geq 0} \biggl(\sum_{k\geq0}(-1)^k \rank \MH_k^\ell(G)\biggr)q^\ell.
\]
In particular, the magnitude homology of $G$ determines $\#G$. 
Furthermore, if $G$ is diagonal, %the magnitude 
$\#G\in\Q(q)$ determines $\rank\MH_k^\ell(G)$. 
\end{theorem}

\subsection{Asao-Izumihara complex}%%%%%%%%%%%%%%%%%%%%%%%%%%%%%%

Asao-Izumihara complex is a CW complex which is obtained as the quotient of 
a simplicial complex $K_\ell(a, b)$ divided by a subcomplex $K'_\ell(a, b)$. 

Let us start with recalling notations on 
simplicial complexes and chain complexes.

Let $V$ be a set and $P(V)$ be the power set of $V$. Recall that $S\subset P(V)\setminus\{\emptyset\}$ is called a (abstract) simplicial complex if it satisfies the following condition:
\begin{itemize}
\item[]{For any $A\in S$ and any subset $B\subseteq A$, $B\neq\emptyset\quad\Rightarrow\quad B\in S$ holds.}
\end{itemize}

Let $X$ be a totally ordered set, and $S\subset P(X)\setminus\{\emptyset\}$ be a simplicial complex.
We denote the cellular chain complex by $C_k(S)$ and the boundary map by $\partial_k$. Namely, 
\[
\begin{split}
&C_k(S):=\bigoplus_{\substack{\{s_0, \cdots, s_k\}\in S, \\s_0<\cdots <s_k}}\mathbb{Z}\langle s_0, \cdots, s_k\rangle, \\
&\partial_k\colon C_k(S)\rightarrow C_{k-1}(S), \ \langle s_0, \cdots, s_k\rangle\mapsto\sum_{i=0}^{k}(-1)^i\langle s_0, \cdots, \widehat{s}_i, \cdots, s_k\rangle.\\
\end{split}
\]

Let $G$ be a graph and $k\geq 1$.
A sequence $(x_0, \cdots, x_k)\in V(G)^{k+1}$ is called a path in $G$ if $d(x_i, x_{i+1})=1$ for any $i\in\{0, \cdots, k-1\}$.
For $a, b\in V(G)$, the set of paths with length $\ell$ which start with $a$ and end with $b$ is denoted by
\[
P_{\ell}(a, b):=\{\, (x_0, \cdots, x_{\ell})=x\colon\text{path in }G\mid x_0=a,\ x_\ell=b,\ L(x)=\ell\,\}.
\]

\begin{definition}[$K_{\ell}(a, b), K'_\ell(a, b)$]
Let $G$ be a graph, and $a, b\in V(G)$, $\ell\geq 3$.
\[
\begin{split}
K_\ell(a, b):=\{\,&\emptyset\neq\{(x_{i_1}, i_1), \cdots, (x_{i_k}, i_k)\}\subset V(G)\times\{1, \cdots, \ell-1\}\\\
&\mid (a, x_{i_1}, \cdots, x_{i_k}, b)\prec\exists (a, x_1, \cdots, x_{\ell-1}, b)\in P_\ell(a, b)\,\}, \\
K'_\ell(a, b):=\{\,&\{(x_{i_1}, i_1), \cdots, (x_{i_k}, i_k)\}\in K_\ell(a, b)\mid L(a, x_{i_1}, \cdots, x_{i_k}, b)\leq \ell-1\,\}.
\end{split}
\]
\end{definition}

\begin{remark}
\begin{itemize}
\item{For simplicity, let us denote $\{(x_{i_1}, i_1), \cdots, (x_{i_k}, i_k)\}$ by $\{x_{i_1}, \cdots, x_{i_k}\}$ when there is no confusion.}
\item{It is easily seen that $K_\ell(a, b)$ is a simplicial complex. Both $K_{\ell-1}(a, b)$ and $K'_\ell(a, b)$ are subcomplexes of $K_\ell(a, b)$, and in general $K_{\ell-1}(a, b)\subsetneq K'_\ell(a, b)$.}
\item{The quotient $K_\ell(a, b)/K'_\ell(a, b)$ is a CW complex, and we call it the Asao-Izumihara complex.}
\end{itemize}
\end{remark}

\begin{proposition}[\cite{AI}, Proposition 2.9.]
Fix $\ell\geq 0$. Let $G$ be a graph, and $a, b\in V(G)$.
Denote the subcomplex of $\MC_*^\ell(G)$ generated by sequences which start with $a$ and end with $b$ by $\MC_*^\ell(a, b)$. Then we have an isomorphism
\[
\MC_*^\ell(G)\cong\bigoplus_{a, b\in V(G)} \MC_*^\ell(a, b)
\]
of chain complexes.
\end{proposition}

\begin{theorem}[\cite{AI}, Theorem 4.3.]\label{thm:AI}
Let $\ell\geq 3$, $*\geq 0$. Then, the isomorphism
\[
(C_*(K_\ell(a, b), K'_\ell(a, b)), -\partial)\cong (\MC_{*+2}^\ell(a, b), \partial)
\]
of chain complexes holds. 
\end{theorem}

\begin{corollary}[\cite{AI}, Corollary 4.4.]\label{cor:AI}
Let $\ell\geq 3$.
\begin{itemize}
\item{In the case of $k\geq 3$, $\MH_k^\ell(a, b)\cong H_{k-2}(K_\ell(a, b), K'_\ell(a, b)). $}
\item{In the case of $k=2$, $\MH_2^\ell(a, b)\cong\begin{cases}H_0(K_\ell(a, b), K'_\ell(a, b))&(d(a, b)<\ell), \\
\tilde{H_0}(K_\ell(a, b))&(d(a, b)=\ell), \end{cases}$
where $\tilde{H_0}(K_\ell(a, b))$ is the reduced homology.}
\end{itemize}
\end{corollary}

\begin{example}
\label{ex:complete}
Let $K_N$ be the complete graph with $N$ vertices. Let $\ell\geq 3$. Fix $a, b\in V(K_N)$. 
Then the Asao-Izumihara complex $K_\ell(a, b)/K'_\ell(a, b)$ is homotopy equivalent to wedge of spheres.
In fact, for any maximal faces of $K_\ell(a, b)$, the boundaries of them are consisting of simplices of $K'_\ell(a, b)$.
\end{example}

\begin{example}
Let $C_4$ be the cycle graph with four vertices $a, b, c, d$ in this order 
(Figure \ref{fig:C_4}). 
%%%%%%%%%%%%%%%%%%%%%%%%%%%%%%%%%%%%%%%%%%%%%%%%%%%
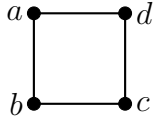
\begin{figure}[htbp]
\centering
\begin{tikzpicture}[scale=1.2]

\filldraw[fill=black, draw=black] (0,0) circle (2pt) ;
\filldraw[fill=black, draw=black] (1,0) circle (2pt) ;
\filldraw[fill=black, draw=black] (0,1) circle (2pt) ;
\filldraw[fill=black, draw=black] (1,1) circle (2pt) ;

\draw[thick] (0,0)--(1,0);
\draw[thick] (0,0)--(0,1);
\draw[thick] (1,1)--(0,1);
\draw[thick] (1,1)--(1,0);

\draw (0,1) node[left]{$a$}; 
\draw (0,0) node[left]{$b$}; 
\draw (1,0) node[right]{$c$}; 
\draw (1,1) node[right]{$d$}; 

\end{tikzpicture}
\caption{The cycle graph $C_4$.}
\label{fig:C_4}
\end{figure}
%%%%%%%%%%%%%%%%%%%%%%%%%%%%%%%%%%
Let $\ell=4$. Let us directly check that the Asao-Izumihara complex 
$K_4(a, a)/K'_4(a, a)$ is homotopy equivalent to wedge of $2$-spheres. 
The maximal faces of $K_4(a, a)$ are as follows.
\[
\begin{split}
&\{(b, 1), (a, 2), (b, 3)\},\ \{(b, 1), (a, 2), (d, 3)\},\ \{(b, 1), (c, 2), (b, 3)\},\\
&\{(b, 1), (c, 2), (d, 3)\},\ \{(d, 1), (a, 2), (b, 3)\},\ \{(d, 1), (a, 2), (d, 3)\},\\
&\{(d, 1), (c, 2), (b, 3)\},\ \{(d, 1), (c, 2), (d, 3)\}.
\end{split}
\]
The subcomplex $K'_4(a, a)$ is as follows.
\[
K'_4(a, a)=\left\{
\begin{array}{l}
\{(b, 1)\},\ \{(a, 2)\},\ \{(b, 3)\},\ \{(d, 3)\},\ \{(d, 1)\},\ \{(a, 2), (b, 3)\},\\
\{(b, 1), (b, 3)\},\ \{(b, 1), (a, 2)\},\ \{(a, 2), (d, 3)\},\ \{(d, 1), (a, 2)\},\\
\{(d, 1), (d, 3)\}
\end{array}
\right\}.
\]
Therefore $K_4(a, a)$ is isomorphic to the octahedron which is homeomorphic to $S^2$, and $K'_4(a, a)$ is homotopy equivalent to wedge of two $1$-spheres as in Figure \ref{fig:AIcpx}. (In Figure \ref{fig:AIcpx}, $K'_4(a, a)$ is drawn in red.)
Then the Asao-Izumihara complex $K_4(a, a)/K'_4(a, a)$ is homotopy equivalent to wedge of three $2$-spheres.
%\end{example}

%%%%%%%%%%%%%%%%%%%%%%%%%%%%%%%%%%%%%%%%%%%%%%%%%%%
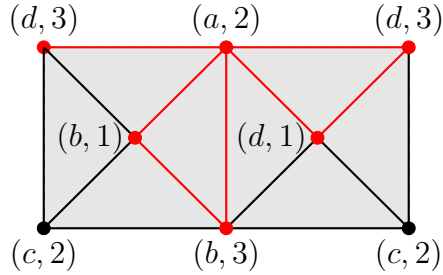
\begin{figure}[htbp]
\centering
\begin{tikzpicture}[scale=1.2]

\filldraw[fill=gray!20!white, draw=black, dashed, very thin] 
(0,0)--(0,2)--(4,2)--(4,0)--cycle;

\filldraw[fill=black, draw=black] (0,0) circle (2pt) ;
\filldraw[fill=red, draw=red] (0,2) circle (2pt) ;
\filldraw[fill=black, draw=black] (4,0) circle (2pt) ;

\draw[thick] (0,0)--(2,0);
\draw[thick] (0,0)--(0,2);
\draw[thick, red] (2,2)--(0,2);
\draw[thick, red] (2,2)--(2,0);
\draw[thick] (4,0)--(4,2);
\draw[thick] (4,0)--(2,0);
\draw[thick, red] (4,2)--(2,2);

\draw[thick] (0,0)--(1,1);
\draw[thick] (0,2)--(1,1);
\draw[thick, red] (2,0)--(1,1);
\draw[thick, red] (2,2)--(1,1);

\draw[thick] (2,0)--(3,1);
\draw[thick, red] (2,2)--(3,1);
\draw[thick] (4,0)--(3,1);
\draw[thick, red] (4,2)--(3,1);

\filldraw[fill=red, draw=red] (2,0) circle (2pt) ;
\filldraw[fill=red, draw=red] (2,2) circle (2pt) ;
\filldraw[fill=red, draw=red] (4,2) circle (2pt) ;
\filldraw[fill=red, draw=red] (1,1) circle (2pt) ;
\filldraw[fill=red, draw=red] (3,1) circle (2pt) ;

\draw (0,2) node[above]{$(d, 3)$}; 
\draw (0,0) node[below]{$(c, 2)$}; 
\draw (1,1) node[left]{$(b, 1)$}; 
\draw (2,0) node[below]{$(b, 3)$}; 
\draw (2,2) node[above]{$(a, 2)$}; 
\draw (3,1) node[left]{$(d, 1)$}; 
\draw (4,0) node[below]{$(c, 2)$}; 
\draw (4,2) node[above]{$(d, 3)$};

\end{tikzpicture}
\caption{$K_4(a, a)$ and $K'_4(a, a)$ associated with $C_4$.}
\label{fig:AIcpx}
\end{figure}
%%%%%%%%%%%%%%%%%%%%%%%%%%%%%%%%%%%%%%%%
\end{example}
%%%%%%%%%%%%%%%%%%%%%%%%%%%%%%%%%%%%%%%%

\subsection{Discrete Morse theory}%%%%%%%%%%%%%%%%%%%
\label{sec:morse}

In this subsection we recall discrete Morse theory to use 
in the proof of Theorem \ref{thm:main}. Our description follows \cite{K}. 
Let $S\subset P(V)$ be a simplicial complex. 
Let us define $a\leq b$ by $a\subseteq b$ for $a, b\in S$. Then $(S, \leq)$ is a poset, 
called the face poset of $S$.

\begin{definition}[partial matching]
Let $P$ be a poset. 
A partial matching $M$ is a subset $M\subseteq P\times P$ satisfying the 
followings.
\begin{itemize}
\item{For any $(a, b)\in M$, $a\prec b$ holds,
 where $a\prec b$ means that $a<b$ and there does not exist any $c\in P$ such that $a<c<b$. }
\item{Each $a\in P$ belongs to at most one element in $M$.}
\end{itemize}
If $(a, b)\in M$, let us denote $a$ by $d(b)$ and $b$ by $u(a)$.
\end{definition}

\begin{definition}[acyclic matching]
A partial matching $M$ on the poset $P$ is said to be acyclic 
if there does not exist a cycle
\begin{equation}
\label{eq:11}
b^1\succ d(b^1)\prec b^2\succ d(b^2)\prec\cdots\prec b^p\succ d(b^p)\prec b^{p+1}=b^1
\end{equation}
with $p\geq 2$, and $b_i\neq b_j$ for every $i, j\in\{1, 2, \cdots, p\}$.
\end{definition}

\begin{definition}[critical element]
Let $M$ be a partial matching on the poset $P$. An element $a\in P$ is called a critical element if $a$ does not belong to any element in $M$.
\end{definition}

\begin{theorem}[\cite{K}, Theorem 11.13.(b)]\label{thm:morse}
Let $P$ be the face poset of a simplicial complex $S$, 
and denote the number of critical $i$-dimensional simplex by $c_i$. 
Then, $S$ is homotopy equivalent to a CW complex with $c_i$ cells in dimension $i$.
\end{theorem}

\section{Main results}%%%%%%%%%%%%%%%%%%%%%%%%%%%%%%%%
\label{sec:main}

%A graph $G$ is said to be diagonal if $\MH_k^\ell(G)=0$ $(k\neq \ell)$. 
%What is noteworthy about diagonal graphs is that a diagonal graph is completely determined by its  magnitude. Since magnitude of a graph $G$ is the sum of elements of the inverse matrix of a matrix which is determined by $(G, d_G)$, magnitude homology of a diagonal graph is easily obtained. 
In this section we study the homotopy type of the Asao-Izumihara complex 
in connection with diagonality. 
Using Corollary \ref{cor:AI}, the diagonality of a graph is characterized 
by the vanishing of homology groups of the Asao-Izumihara complex. 
We have the following.

\begin{proposition}\label{prop:diagonal}
Let $G$ be a graph. The following are equivalent.
\begin{itemize}
\item[$(\mathrm{I})$]{The graph $G$ is diagonal.}
\item[$(\mathrm{II})$]{$\tilde{H}_k(K_\ell(a, b)/K'_\ell(a, b))=0$ $(\ell\geq 3,\ k\neq \ell-2)$ for any $a, b\in V(G).$}
\end{itemize}
\end{proposition}

%\color{red}%%%%%%%%%%%%%%%%%%%%%%%
\begin{proof}
Suppose the condition ($\mathrm{II}$) holds. We will show that $G$ is diagonal, i.e. $\MH_k^\ell(G)=0$ for $k\neq\ell$.
\begin{itemize}
\item{In the case of $\ell\geq 3$ and $k\geq 3$, by Corollary \ref{cor:AI}
\[
\MH_k^\ell(a, b)\Bigl(\cong H_{k-2}(K_\ell(a, b), K'_\ell(a, b))\Bigr)\cong\tilde{H}_{k-2}(K_\ell(a, b)/K'_\ell(a, b))=0
\]
for $k\neq\ell$ and any $a, b\in V(G)$. Then $\MH_k^\ell(G)=0$ for $k\neq\ell$.
}
\item{In the case of $\ell\geq 3$ and $k=2$, by Corollary \ref{cor:AI} 
\[
\MH_2^\ell(a, b)\cong\begin{cases}H_0(K_\ell(a, b), K'_\ell(a, b))\cong \tilde{H}_0(K_\ell(a, b)/K'_\ell(a, b))=0&(d(a, b)<\ell)\\
\tilde{H_0}(K_\ell(a, b))\cong\tilde{H_0}(K_\ell(a, b)/K'_\ell(a, b))=0&(d(a, b)=\ell).\end{cases}
\]
for any $a, b\in V(G)$. Then $\MH_k^\ell(G)=0$ for $k\neq\ell$.}
\item{In the case of $\ell\geq 2$ and $k=1$, $\MH_1^\ell(G)=0$ also holds since any chain $(x_0, x_1)\in \Ker\partial_1=\MC_1^\ell(G)$ is in $\Im\partial_2$. }
\end{itemize}
Note that $\MH_0^\ell(G)=0$ for any graphs and $\ell\geq 1$. Then the condition ($\mathrm{I}$) holds. 
The proof of ($\mathrm{I})\Rightarrow (\mathrm{II})$ is obvious.
\end{proof}
%\color{black}%%%%%%%%%%%%%%%%%%%%%%

Furthermore, we also have the following.

\begin{proposition}\label{prop:wedge}
Let $G$ be a graph. If for any $a, b\in V(G)$, the Asao-Izumihara complex $K_\ell(a, b)/K'_\ell(a, b)$ is empty or contractible or homotopy equivalent to 
wedge of $(\ell-2)$-dimensional spheres, then the graph $G$ is diagonal.
\end{proposition}

%\color{red}%%%%%%%%%%%%%%%%%%%%%%%%%
\begin{proof}
Suppose $K_\ell(a, b)/K'_\ell(a, b)$ is as above. Then, 
\[
\tilde{H}_k(K_\ell(a, b)/K'_\ell(a, b))=0
\]
for $\ell\geq 3$, $k\neq \ell-2$ and any $a, b\in V(G)$. By Proposition \ref{prop:diagonal}, $G$ is diagonal. 
\end{proof}

%\color{black}%%%%%%%%%%%%%%%%%%%%%%%%%%%%%%%%
Pawful graphs are defined in \cite{G} as a new class of diagonal graphs. 
For example, the join of non-empty graphs are pawful graphs (complete graphs are contained in it). We already know that the Asao-Izumihara complexes for complete graphs are homotopy equivalent to wedge of spheres (Example \ref{ex:complete}). One of the purposes of this paper is to show that the Asao-Izumihara complexes for pawful graphs are homotopy equivalent to wedge of spheres. 

\begin{definition}[pawful graph]
\label{def:pawful}
A (connected) graph $G$ is called a pawful graph if it satisfies the following conditions.
\begin{itemize}
\item{$d(x, y)\leq 2$ for any $x, y\in V(G)$, }
\item{for any $x, y, z\in V(G)$ with $d(x, y)=d(y, z)=2$ and $d(x, z)=1$, there exists $a\in V(G)$ such that
 $d(a, x)=d(a, y)=d(a, z)=1$.}
\end{itemize}
\end{definition}
The following is the main result. 
\begin{theorem}
\label{thm:main}
Let $G$ be a pawful graph. Fix $\ell\geq 3$ and $a, b\in V(G).$ Then, the Asao-Izumihara complex 
$K_\ell(a, b)/K'_\ell(a, b)$ is empty or contractible or homotopy equivalent to wedge of $(\ell-2)$-spheres.
\end{theorem}

By Theorem \ref{thm:main} and Proposition \ref{prop:wedge}, 
we have the following. 

\begin{corollary}[\cite{G}, Theorem 4.4.]
Pawful graphs are diagonal.
\end{corollary}

\section{Proof}%%%%%%%%%%%%%%%%%%%%%%%%%%%%%%%%%%%%%%%%%%%%%%%%%
\label{sec:proofs}

To prove the main theorem (Theorem \ref{thm:main}), 
we will construct an acyclic matching on 
$K_\ell(a, b)\setminus K'_\ell(a, b)$. Recall that a pair consists of a lower 
dimensional simplex and a higher dimensional simplex. 
When we specify a pair, we obtain the lower dimensional one by 
removing a smooth point from a higher dimensional one. 
A naive such removal makes the pairing neither matching nor acyclic. 
We need to take care to make the paring both matching and acyclic. 

\begin{proof}[Proof of Theorem \ref{thm:main}]
The Theorem is proved by using discrete Morse theory. For this purpose, we construct an acyclic matching on $K_\ell(a, b)/K'_\ell(a, b)$.
Assume that $K_\ell(a, b)\neq \emptyset$. 
Let $P$ be the face poset of consisting of simplices of 
$K_\ell(a, b)$ which is not contained in $K'_\ell(a, b)$. 
We will construct acyclic matching on $P$ such that all critical simplices are 
$(\ell-2)$-dimensional. 
By definition of pawful graphs, there exists a map $p\colon\{ (x, y, z)\in V(G)^3\mid d(x, y)=d(y, z)=2,\ d(x, z)=1\}\rightarrow V(G)$ such that $v=p(x, y, z)$ satisfies 
$d(v, x)=d(v, y)=d(v, z)=1$. Also there exists a map $q\colon\{(x, y)\in V(G)^2\mid d(x, y)=2\}\rightarrow V(G)$ 
such that $w=q(x, y)$ satisfies $d(x, w)=d(w, y)=1$. 
Next, we define the set $\scS$ as follows.
\begin{equation}
\begin{split}
\scS:=&
\left\{
(\alpha, \beta, \gamma, \delta)\in V(G)^4\left| 
d(\alpha, \delta)=1,\ d(\beta, \delta)=2,\ d(\alpha, \beta)=1,\ \gamma=\alpha
\right.\right\}
\\
&\sqcup
\left\{(\alpha, \beta, \gamma, \delta)\in V(G)^4\left|
\begin{array}{l}
d(\alpha, \delta)=2,\ d(\beta, \delta)=2,\\
d(\alpha, \beta)=1,\ \gamma=p(\alpha, \delta, \beta)
\end{array}
\right.
\right\}
\\
&\sqcup
\{(\beta, \gamma, \delta)\in V(G)^3\mid d(\beta, \delta)=2,\ \gamma=q(\beta, \delta)\}. 
\end{split}
\end{equation}
Let $(\alpha, \beta, \gamma, \delta)\in \scS$. Then $\gamma$ is smooth, and $d(\alpha, \gamma)\leq 1$. 
For $x=(x_0, x_1, \cdots, x_k)\in P$, we define $i(x)$ and $j(x)$ as follows. If there exist any $i\in \{ 1, \cdots, k-2\}$ with $(x_{i-1}, x_i, x_{i+1}, x_{i+2})\in \scS$, denote the minimum of $i$ by $i(x)$. 
However in the case of $(x_0, x_1, x_2)\in \scS$, define $i(x)=0$.
If there does not exist such a $i\in \{ 1, \cdots, k-2\}$ and $i(x)\neq 0$, we define $i(x)=\infty$. Also, if there exist any $j\in \{ 0, \cdots, k-1\}$ with $d(x_j, x_{j+1})=2$, denote the minimum of $j$ by $j(x)$. If there does not exist such a $j\in \{ 0, \cdots, k-1\}$, we define $j(x)=\infty$. Next we define subsets $A, P', P''\subseteq P$ as follows.
\[
\begin{split}
A&:=\{x\in P\mid i(x)=\infty\text{ and }j(x)=\infty\},\\
P'&:=\{x\in P\mid i(x) > j(x)\}, \\
P''&:=\{x\in P\mid i(x) < j(x)\}.
\end{split}
\]
Define a map $f\colon P'\rightarrow P''$ by 
\[
x=(x_0, \cdots, x_k)\mapsto(x_0, \cdots, x_{j(x)-1}, x_{j(x)}, z, x_{j(x)+1}, \cdots, x_k ), 
\]
where $z$ satisfies $(x_0, z, x_1)\in \scS$ when $j(x)=0$, 
or $(x_{j(x)-1}, x_{j(x)}, z, x_{j(x)+1})\in \scS$ when $j(x)\geq 1$. 
Define $g\colon P''\rightarrow P'$ by $y=(y_0, \cdots, y_k)\mapsto (y_0, \cdots, \widehat{y_{i(y)+1}}, \cdots, y_k)$.
Then,  for $y=(y_0, \cdots, y_k)\in P''$, 
\[
\begin{split}
f\circ g((y_0, \cdots, y_k))&=f((y_0, \cdots, y_{i(y)}, y_{i(y)+2}, \cdots, y_k))\\
&=(y_0, \cdots, y_{i(y)}, z, y_{i(y)+2}, \cdots, y_k)\\
&=(y_0, \cdots, y_k).
\end{split}
\]
On the other hand, for $x=(x_0, \cdots, x_k)\in P'$, 
\[
g\circ f((x_0, \cdots, x_k))=g((x_0, \cdots, x_{j(x)}, z, x_{j(x)+1}, \cdots, x_k))=:(*).
\]
For $j(x)\geq 2$, $(x_{j(x)-2}, x_{j(x)-1}, x_{j(x)}, z)\notin \scS$ holds since $d(x_{j(x)-1}, z)\leq 1$. 
Also for $j(x)=1$, $(x_0, x_1, z)\notin \scS$ holds since $d(x_0, z)\leq 1$. 
Hence $(*)=(x_0, \cdots, x_k).$ 
Therefore, since $f\circ g=\id_{P''}$ and $g\circ f=\id_{P'}$, $f$ is bijection. 
Define a matching $M\subseteq P\times P$ by
\[
M:=\{(x, y)\in P\times P\mid x\in P',\ y\in P'',\ f(x)=y\}.
\]
Then, any $(x, y)\in M$ satisfies $x\prec y$. Since $f$ is bijection and $P'\cap P''=\emptyset$, 
each $x\in P$ belongs to at most one element in $M$. 
It remains to prove that $M$ is acyclic. 

Suppose that there exists the following cycle
\[
y^1\succ x^1\prec y^2\succ x^2\prec\cdots \succ x^p\prec y^{p+1}=y^1,
\]
where $y^{t+1}=f(x^t)\ (t\geq 1)$ and, 
$x^t$ is the removal of a smooth point $y_{i^t+1}^t$ from 
$y^t=(y_0^t, \cdots, y_k^t)\in P''$. 
By definition, we have $i(x^t)>j(x^t)$ and $i(y^t)<j(y^t)$. 
For $x^t\in P'$, let $j^t:=j(x^t)$. 
We may assume that $j^1$ is the minimum of all $j^t$. Then, the followings hold.
\begin{itemize}
\item{$j^t=i^t\ (t\geq 1)$.}
\item{$j^t+1\geq j^{t+1}$, $j^t\neq j^{t+1}\ (t\geq 1)$.}
\item{$j^2=j^1+1$ (because of the minimality of $j^1$).}
\end{itemize}
Let $y^1=(y_0, \cdots, y_k)$, then $x^1, y^2, x^2, y^3$ are as follows. 
\[
\begin{split}
x^1&=(y_0, \cdots, y_{i^1}, y_{i^1+2}, \cdots, y_k),\\
y^2&=(y_0, \cdots, y_{i^1}, z_1, y_{i^1+2}, \cdots, y_k)\ (z_1\text{ satisfies }(y_{i^1-1}, y_{i^1}, z_1, y_{i^1+2})\in \scS),\\
x^2&=(y_0, \cdots, y_{i^1}, z_1, y_{i^1+3}, \cdots, y_k),\\
y^3&=(y_0, \cdots, y_{i^1}, z_1, z_2, y_{i^1+3}, \cdots, y_k)\ (z_2\text{ satisfies }(y_{i^1}, z_1, z_2, y_{i^1+3})\in \scS).\\
\end{split}
\]
Let $(y_0^t, \cdots, y_k^t):=y^t$. We will show 
$d(y_{i^1}^t, y_{i^1+2}^t)\leq 1\ (t\geq 3)$ by induction.
\begin{itemize}
\item[$(\mathrm{I})$]{
In the case of $t=3$, 
for $y^3=(y_0, \cdots, y_{i^1}, z_1, z_2, y_{i^1+3}, \cdots, y_k)$, $(y_{i^1}, z_1, z_2, y_{i^1+3})\in \scS$ holds. Then either of
\begin{itemize}
\item[$\bullet$]{$z_2=y_{i^1}$}
\item[$\bullet$]{$z_2=p(y_{i^1}, z_1, y_{i^1+3})$}
\end{itemize}
is true, but in any case, $d(y_{i^1}^3, y_{i^1+2}^3)=d(y_{i^1}, z_2)\leq 1$ holds.
}
\item[$(\mathrm{II})$]{
For $y^t\ (t\geq 3)$, assume that $d(y_{i^1}^t, y_{i^1+2}^t)\leq 1$. Then, 
\[
y^{t+1}=(y_0^t, \cdots, y_{i^t-1}^t, y_{i^t}^t, z_t, y_{i^t+2}^t, \cdots, y_k^t)
 (z_t\text{ satisfies} (y_{i^t-1}^t, y_{i^t}^t, z_t, y_{i^t+2}^t)\in \scS).
\]
From $d(y_{i^1}^t, y_{i^1+2}^t)\leq 1$, we see $i^t\neq i^1$. Combining with $i^t\geq i^1$, 
we have $i^t\geq i^1+1$.
\begin{itemize}
\item[$(\mathrm{i})$]{
In the case of $i^t=i^1+1$, 
in the same way as $(\mathrm{I})$, by $(y_{i^t-1}^t, y_{i^t}^t, z_t, y_{i^t+2}^t)\in \scS$, 
$d(y_{i^1}^{t+1}, y_{i^1+2}^{t+1})=d(y_{i^t-1}^t, z_t)\leq 1$ is valid. 
}
\item[$(\mathrm{ii})$]{
In the case of $i^t\geq i^1+2$, 
$d(y_{i^1}^{t+1}, y_{i^1+2}^{t+1})=d(y_{i^1}^t, y_{i^1+2}^t)\leq 1.$
}
\end{itemize}
}
\end{itemize}
From $(\mathrm{I})$ and $(\mathrm{II})$, we obtain $d(y_{i^1}^t, y_{i^1+2}^t)\leq 1\ (t\geq 3)$. 
Since $y_{i^1+1}^t$ is not smooth point on $y^t$, $j^t=i^t\neq i^1=j^1\ (i\geq 3)$.
Combining with $j^2\neq j^1$, $j^t\neq j^1$ $(t\geq 2)$ holds.
On the other hand, from $y^1=y^{p+1}$, we have $j^1=i^1=i^{p+1}=j^{p+1}$, but it contradictis $j^t\neq j^1\ (t\geq 2)$.
Now we can show that $M$ is an acyclic matching. When considering a matching $M$ on the face poset of $K_\ell(a, b)$, then the critical elements are only elements of $A$ or elements of $K'_\ell(a, b)$. 
\begin{itemize}
\item{If $A=\emptyset$, then $K_\ell(a, b)$ is homotopy equivalent to $K'_\ell(a, b)$. Therefore $K_\ell(a, b)/K'_\ell(a, b)$ is contractible. }
\item{If $A\neq\emptyset$, then all elements of $A$ are $(\ell-2)$-simplexes since they satisfies $j_0=\infty$. By Theorem \ref{thm:morse}, $K_\ell(a, b)$ is homotopy equivalent to a CW complex with some $(\ell-2)$-cell attached to $K'_\ell(a, b)$. Therefore, $K_\ell(a, b)/K'_\ell(a, b)$ is homotopy equivalent to wedge of $(\ell-2)$-dim spheres.}
\end{itemize}
\end{proof}

\section{Generalization}%%%%%%%%%%%%%%%%%%%%%%%%%%%%%%%%
\label{sec:general}

In the previous section, we proved that 
the Asao-Izumihara complex is homotopy equivalent to a wedge of spheres 
for pawful graphs. In this section, we show 
that there are 
many non-pawful diagonal graphs with diameter $2$. 
Indeed, we introduce a slightly wider class of graphs 
which are diagonal. 
Note that the key fact in the proof of the main theorem 
(\S\ref{sec:proofs}) is the following property for a pawful graph $G$. 
\begin{itemize}
\item[$(\star)$] 
For any $(\alpha, \beta, \delta)\in V(G)^3$ with $d(\alpha, \beta)=1$ and $d(\beta, \delta)=2$, there exists $\gamma\in V(G)$ such that $d(\alpha, \gamma)\leq 1$ and $d(\beta, \gamma)=d(\gamma, \delta)=1$. 
\end{itemize}
For a non-pawful graph $G$, the property $(\star)$ does not hold in general. 
There may exist $(\alpha, \beta, \delta)\in V(G)^3$ with $d(\alpha, \beta)=1$ and $d(\beta, \delta)=2$ such that any vertex $\gamma\in V(G)$ with $d(\beta, \gamma)=d(\gamma, \delta)=1$ always satisfies $d(\alpha, \gamma)=2$.
In this section, we will see that even such cases, we can construct an acyclic matching $M$ if $G$ satisfies an extra condition.\\

Let $G$ be a graph such that $d(x, y)\leq 2$ for any $x, y\in V(G)$. 
Define $X, Y, X'$, and $Y'$ as follows. 
\[
\begin{split}
X&:=\{(a, b, c)\in V(G)^3\mid d(a, b)=d(b, c)=1,\ d(a, c)=2\}, \\
Y&:=\{(\alpha, \beta, \gamma, \delta)\in V(G)^4\mid d(\alpha, \beta)=d(\beta, \gamma)=d(\gamma, \delta)=1,\ d(\beta, \delta)=2\}, \\
X'&:=\{(a, c)\in V(G)^2\mid d(a, c)=2\}, \\
Y'&:=\{(\alpha, \beta, \delta)\in V(G)^3\mid d(\alpha, \beta)=1,\ d(\beta, \delta)=2\}.\\
\end{split}
\]

Then, there exist natural projections $\pi\colon X\rightarrow X'$ and $\pi\colon Y\rightarrow Y'$.

\begin{definition}
\label{def:S}
Assume that maps $f_1\colon X'\rightarrow X$ and $f_2\colon Y'\rightarrow Y$ satisfy the following conditions.
\begin{itemize}
\item[$(\mathrm{i})$]{
$\pi\circ f_1=\id_{X'}$, $\pi\circ f_2=\id_{Y'}$.
}
\item[$(\mathrm{ii})$]{
Let $(\alpha, \beta, \gamma, \delta)\in f_2(Y')$. Then $(\ast, \alpha, \beta, \gamma)\notin f_2(Y')$ and 
$(\alpha, \beta, \gamma)\notin f_1(X')$.}
\item[$(\mathrm{iii})$]{
Let $(\alpha, \beta, \gamma, \delta)\in f_2(Y')$. 
If $d(\alpha, \gamma)=2$, then there does not exist $\gamma'(\neq\gamma)\in V(G)$ such that 
$d(\beta, \gamma')=d(\gamma', \delta)=1$.
}
\end{itemize}
Define the set $\scS\subseteq X\sqcup Y$ by $\scS :=f_1(X')\sqcup f_2(Y')$.
\end{definition}

\begin{theorem}
\label{thm:generalization}
Suppose that $G$ has $\scS$ as in Definition \ref{def:S}. Then 
the Asao-Izumihara complex is homotopy equivalent to 
a wedge of spheres. 
In particular, $G$ is diagonal. 
\end{theorem}

\begin{proof}
Similar to  Theorem \ref{thm:main}. 
\end{proof}

\begin{remark}
\label{rem:proof}
As was noted, the condition ($\mathrm{iii}$) in Definition \ref{def:S} is a 
new condition. 
Let us point out at which point the new condition is used in 
the proof of Theorem \ref{thm:generalization}. 
The construction of the matching $M$ by using the set $\scS$ is similarly to 
the proof of Theorem \ref{thm:main}. The crucial point is the inequality 
$d(y_{i^1}^t, y_{i^1+2}^t)\leq 1\ (t\geq 3)$ in the proof of Theorem \ref{thm:main}. 
In the proof of Theorem \ref{thm:generalization}, this inequality 
is obtained thanks to the condition ($\mathrm{iii}$). 
\end{remark}

\begin{example}
\label{ex:01}
Let $G_1$ be a graph as in Figure \ref{fig:non-pawful}. 
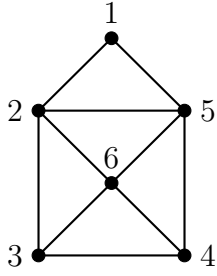
\begin{figure}[htbp]
\centering
\begin{tikzpicture}[scale=1.2]

\filldraw[fill=black, draw=black] (0.8, 2.4) circle (2pt) ;
\filldraw[fill=black, draw=black] (0, 1.6) circle (2pt) ;
\filldraw[fill=black, draw=black] (0, 0) circle (2pt) ;
\filldraw[fill=black, draw=black] (1.6, 0) circle (2pt) ;
\filldraw[fill=black, draw=black] (1.6, 1.6) circle (2pt) ;
\filldraw[fill=black, draw=black] (0.8, 0.8) circle (2pt) ;

\draw[thick] (0.8, 2.4)--(0, 1.6);
\draw[thick] (0, 1.6)--(0, 0);
\draw[thick] (0, 0)--(1.6, 0);
\draw[thick] (1.6, 0)--(1.6, 1.6);
\draw[thick] (1.6, 1.6)--(0.8, 2.4);
\draw[thick] (0, 1.6)--(1.6, 1.6);
\draw[thick] (0, 1.6)--(0.8, 0.8);
\draw[thick] (0, 0)--(0.8, 0.8);
\draw[thick] (1.6, 0)--(0.8, 0.8);
\draw[thick] (1.6, 1.6)--(0.8, 0.8);

\draw (0.8, 2.45) node[above]{$1$}; 
\draw (-0.05, 1.6) node[left]{$2$}; 
\draw (-0.05, 0) node[left]{$3$}; 
\draw (1.65, 0) node[right]{$4$}; 
\draw (1.65, 1.6) node[right]{$5$}; 
\draw (0.8, 0.85) node[above]{$6$}; 

\end{tikzpicture}
\caption{Non-pawful diagonal graph $G_1$.}
\label{fig:non-pawful}
\end{figure}
Note that $G_1$ is not pawful, because for vertices $1, 3, 4$ with $d(1, 3)=d(1, 4)=2$ and $d(3, 4)=1$ there does not exist a vertix $a$ such that $d(a, 1)=d(a, 3)=d(a, 4)=1$. For this graph, we can construct the set $S$ as follows. Define the map $f_1\colon X'\rightarrow X$ by
\[
f_1(X')=\left\{
\begin{array}{l}
(1, 2, 3), (1, 5, 4), (1, 2, 6), (2, 6, 4), (3, 2, 1), (3, 6, 5), \\
(4, 5, 1), (4, 6, 2), (5, 6, 3), (6, 2, 1)
\end{array}
\right\}.
\]
Define the map $f_2\colon Y'\rightarrow Y$ by
\[
f_2(Y')=\left\{
\begin{array}{l}
(2, 1, 2, 3), (5, 1, 2, 3), (2, 1, 5, 4), (5, 1, 5, 4), (2, 1, 2, 6), (5, 1, 2, 6), \\
(1, 2, 5, 4), (3, 2, 6, 4), (5, 2, 6, 4), (6, 2, 6, 4), \\
(2, 3, 2, 1), (4, 3, 2, 1), (6, 3, 2, 1), (2, 3, 6, 5), (4, 3, 6, 5), (6, 3, 6, 5), \\
(3, 4, 5, 1), (5, 4, 5, 1), (6, 4, 5, 1), (3, 4, 6, 2), (5, 4, 6, 2), (6, 4, 6, 2), \\
(1, 5, 2, 3), (2, 5, 6, 3), (4, 5, 6, 3), (6, 5, 6, 3), \\
(2, 6, 2, 1), (3, 6, 2, 1), (4, 6, 5, 1), (5, 6, 5, 1)
\end{array}
\right\}.
\]
Then, since we have 
\[
\{(\alpha, \beta, \gamma, \delta)\in f_2(Y')\mid d(\alpha, \gamma)=2\}
=\{
(4, 3, 2, 1), (3, 4, 5, 1)
\}, 
\]
it is easily seen that $f_1$ and $f_2$ satisfy the conditions $(\mathrm{i})$, $(\mathrm{ii})$ and $(\mathrm{iii})$ of Definition \ref{def:S}. 
Thus we conclude that $G_1$ is diagonal. 
%Thus, we obtain $\scS :=f_1(X')\sqcup f_2(Y')$.
The magnitude of $G_1$ is directly computed as 
\[
\begin{split}
\#G_1&=\frac{2q^3+4q^2-10q-6}{q^5+q^4-6q^2-5q-1}\\
&=6-20q+60q^2-182q^3+556q^4-1702q^5+5214q^6-15980q^7+\cdots. 
\end{split}
\]
From the diagonality, we have 
\[
\MH_0^0(G_1)\cong\Z^6,\ \MH_1^1(G_1)\cong\Z^{20}, \ \MH_3^3(G_1)\cong\Z^{60}, \cdots, 
\]
and others are vanishing.
%Non-pawful diagonal graphs. ($C_4$ plus 4 triangles)

%%%%%%%%%%%%%%%%%%%%%%%%%%%%%%%%%%%%%%%%%%%%%%%%%%%
%%%%%%%%%%%%%%%%%%%%%%%%%%%%%%%%%%%%%%%%%%%%%%%%%%%
\end{example}

\begin{example}
\label{ex:02} 
Let $G_2$ be a graph as in Figure \ref{fig:another}. 
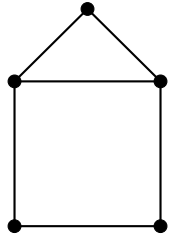
\begin{figure}[htbp]
\centering
\begin{tikzpicture}[scale=1.2]

\filldraw[fill=black, draw=black] (0.8, 2.4) circle (2pt) ;
\filldraw[fill=black, draw=black] (0, 1.6) circle (2pt) ;
\filldraw[fill=black, draw=black] (0, 0) circle (2pt) ;
\filldraw[fill=black, draw=black] (1.6, 0) circle (2pt) ;
\filldraw[fill=black, draw=black] (1.6, 1.6) circle (2pt) ;

\draw[thick] (0.8, 2.4)--(0, 1.6);
\draw[thick] (0, 1.6)--(0, 0);
\draw[thick] (0, 0)--(1.6, 0);
\draw[thick] (1.6, 0)--(1.6, 1.6);
\draw[thick] (1.6, 1.6)--(0.8, 2.4);
\draw[thick] (0, 1.6)--(1.6, 1.6);

\end{tikzpicture}
\caption{Diagonal graph $G_2$ which does not have $\scS$.}
\label{fig:another}
\end{figure}
Note that $G_2$ is not pawful, and 
diagonal graph (this fact can be checked by using 
Mayer-Vietoris sequence in \cite{HW}. See also the table below). 
However, there does not exist $\scS$ as in Definition \ref{def:S}. 
%An example of such graph is shown in Figure \ref{fig:another}.
The rank of $\MH_k^\ell(G_2)$ is as follows. 
\[
\begin{array}{c|c|c|c|c|c|c|c}
\ell\backslash k&0&1&2&3&4&5&6\\
\hline
0&5&&&&&&\\
\hline
1&&12&&&&&\\
\hline
2&&&22&&&&\\
\hline
3&&&&38&&&\\
\hline
4&&&&&66&&\\
\hline
5&&&&&&118&\\
\hline
6&&&&&&&218
\end{array}
\]
\end{example}

\begin{example}
\label{ex:03} 
Let $G_3$ be a graph as in Figure \ref{fig:non-diag}. 
\begin{figure}[htbp]
\centering
\begin{tikzpicture}[scale=1.2]

\filldraw[fill=black, draw=black] (0.8, 2.4) circle (2pt) ;
\filldraw[fill=black, draw=black] (0, 1.6) circle (2pt) ;
\filldraw[fill=black, draw=black] (0, 0) circle (2pt) ;
\filldraw[fill=black, draw=black] (1.6, 0) circle (2pt) ;
\filldraw[fill=black, draw=black] (1.6, 1.6) circle (2pt) ;
\filldraw[fill=black, draw=black] (0.8, 0.8) circle (2pt) ;

\draw[thick] (0.8, 2.4)--(0, 1.6);
\draw[thick] (0, 1.6)--(0, 0);
\draw[thick] (0, 0)--(1.6, 0);
\draw[thick] (1.6, 0)--(1.6, 1.6);
\draw[thick] (1.6, 1.6)--(0.8, 2.4);
%\draw[thick] (0, 1.6)--(1.6, 1.6);
\draw[thick] (0.8, 2.4)--(0.8, 0.8);
\draw[thick] (0, 0)--(0.8, 0.8);
\draw[thick] (1.6, 0)--(0.8, 0.8);
%\draw[thick] (1.6, 1.6)--(0.8, 0.8);

%\draw (0.8, 2.45) node[above]{$1$}; 
%\draw (-0.05, 1.6) node[left]{$2$}; 
%\draw (-0.05, 0) node[left]{$3$}; 
%\draw (1.65, 0) node[right]{$4$}; 
%\draw (1.65, 1.6) node[right]{$5$}; 
%\draw (0.8, 0.85) node[left]{$6$}; 

\end{tikzpicture}
\caption{Non-diagonal graph $G_3$ such that each edge is contained in a cycle 
of length $\leq 4$.}
\label{fig:non-diag}
\end{figure}
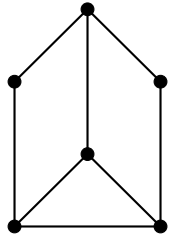
Note that each edge is 
contained in a cycle of length $\leq 4$. However, it is not diagonal. 
The rank of $\MH_k^\ell(G_3)$ is as follows. 
\[
\begin{array}{c|c|c|c|c|c|c|c}
\ell\backslash k&0&1&2&3&4&5&6\\
\hline
0&6&&&&&&\\
\hline
1&&16&&&&&\\
\hline
2&&&30&&&&\\
\hline
3&&&2&50&&&\\
\hline
4&&&&10&82&&\\
\hline
5&&&&&28&138&\\
\hline
6&&&&&2&60&242
\end{array}
\]
\end{example}

\medskip

\noindent
{\bf Acknowledgements.} 
Masahiko Yoshinaga 
was partially supported by JSPS KAKENHI 
Grant Numbers JP19K21826, JP18H01115. 
We thank the referee(s) for their lots of comments 
to improve the paper. 

\end{document}